\documentclass[amscd,amssymb,verbatim,10pt]{amsart}

\newtheorem{theorem}{Theorem}[section]
\newtheorem{lemma}[theorem]{Lemma}

\theoremstyle{definition}

\theoremstyle{remark}

\begin{document}

\title{The action of Hecke operators on hypergeometric functions}

\author{Victor H. Moll}
\address{Department of Mathematics,
Tulane University, New Orleans, LA 70118}
\email{vhm@math.tulane.edu}

\author{Sinai Robins}
\address{Department of Mathematics,
Temple University, Philadelphia, PA 19122}
\email{rsinai@nmath.temple.edu}
\address{Division of Mathematical Sciences,
Nanyang Technological University, Singapore }
\email{rsinai@ntu.edu.sg}

\author{Kirk Soodhalter}
\address{Department of Mathematics,
Temple University, Philadelphia, PA 19122}
\email{ksoodha@temple.edu}


\date{\today}


\begin{abstract}
We study the action of the Hecke operators $U_n$ on the set of hypergeometric 
functions, as well as on formal power series.  We show that the spectrum of 
these operators on the set of hypergeometric functions is the set 
$\{  n^a :   n \in \mathbb{N}  \text{ and } a \in \mathbb{Z}  \}$, and that the 
polylogarithms play a dominant role in the study of the eigenfunctions of 
the Hecke operators $U_n$ on the set of hypergeometric functions.   As a 
corollary of our  results on simultaneous eigenfunctions, we also obtain 
an apriori unrelated result regarding the behavior of completely 
multiplicative hypergeometric coefficients.
\end{abstract}

\maketitle

\newcommand{\nn}{\nonumber}
\newcommand{\ba}{\begin{eqnarray}}
\newcommand{\ea}{\end{eqnarray}}
\newcommand{\E}{{\mathfrak{E}}}
\newcommand{\Ro}{{\mathfrak{R}}}
\newcommand{\ift}{\int_{0}^{\infty}}
\newcommand{\ifft}{\int_{- \infty}^{\infty}}
\newcommand{\no}{\noindent}
\newcommand{\X}{{\mathbb{X}}}
\newcommand{\Q}{{\mathbb{Q}}}
\newcommand{\Y}{{\mathbb{Y}}}
\newcommand{\Ftwo}{{{_{2}F_{1}}}}
\newcommand{\realpart}{\mathop{\rm Re}\nolimits}
\newcommand{\imagpart}{\mathop{\rm Im}\nolimits}
\newcommand{\F}{{\mathfrak{F}}}
\newcommand{\fjpq}{{\mathfrak{F}_{(p,q)}^{j}}}
\newcommand{\fmod}{{\mathfrak{F}_{\text{mod}}}}
\newcommand{\fext}{{\mathfrak{F}_{\text{hypergeom}}}}
\newcommand{\hyp}{{\mathfrak{H}}}
\newcommand{\MK}{{{\mathfrak{M}}_{k}}}
\newcommand{\R}{{\mathfrak{R}}}
\newcommand{\spun}{{\text{Spec}(U_{n})}}
\newcommand{\isum}[1]{\sum_{k=#1}^{\infty}}
\newcommand{\ZZ}{\mathbb{Z}}
\newcommand{\CC}{\mathbb{C}}
\newcommand{\N}{\mathbb{N}}

\newtheorem{Definition}{\bf Definition}[section]
\newtheorem{Thm}[Definition]{\bf Theorem}
\newtheorem{Example}[Definition]{\bf Example}
\newtheorem{Lem}[Definition]{\bf Lemma}
\newtheorem{Note}[Definition]{\bf Note}
\newtheorem{Cor}[Definition]{\bf Corollary}
\newtheorem{Prop}[Definition]{\bf Proposition}
\newtheorem{Problem}[Definition]{\bf Problem}
\numberwithin{equation}{section}

\maketitle

\section{Introduction} \label{S:intro} 

For each $n \in \mathbb{N}$, the space of formal power series 
\begin{equation}
\F := \left\{ 
f(x) = \sum_{k=0}^{\infty} c_{k}x^{k}: \, c_{k} \in \mathbb{C} 
\right \} 
\label{formal-def}
\end{equation}
\noindent
admits  the action of the linear operators 
\begin{equation}
\left( U_{n}f \right)(x) := \sum_{k=0}^{\infty} c_{nk}x^{k} 
\label{un-def}
\end{equation}
\noindent
and 
\begin{equation}
\left( V_{n}f \right)(x) := f(x^{n}) = \sum_{k=0}^{\infty} c_{k}x^{nk}.
\label{vn-def}
\end{equation}
The spectral properties of these operators become more interesting when 
one considers their action on spaces with additional structure.  

Historically, Hecke studied the vector spaces of modular forms of a fixed 
weight \cite{Apostol}, in which the set 
$\F$ is replaced by the space $\MK$ defined by analytic functions in the 
upper half-plane ${\mathbb{H}} := \{ \tau: \imagpart{\tau} > 0 \}$, that 
satisfy the condition
\begin{equation}
f \left( \frac{a \tau + b}{c \tau + d } \right) = (c \tau + d)^{k} f( \tau),
\end{equation}
\noindent 
for every matrix in the modular group 
\begin{equation}
\Gamma := \left\{ \begin{pmatrix} a & b \\ c & d \end{pmatrix}: \, 
a, \, b, \, c, \, d \in {\mathbb{Z}} \text{ with } ad-bc = 1  \right\}.
\end{equation}
\noindent 
These forms are also required 
to have an expansion at $\tau = i \infty$, or equivalently a Taylor series 
about $q=0$: 
\begin{equation}
f( \tau)  = \sum_{n=0}^{\infty} c(n) q^{n},
\end{equation}
expressed in terms of the parameter
$q = \text{exp}( 2 \pi i \tau)$. 
\noindent 

Hecke introduced a family of operators $T_{n}$, for $n \in \mathbb{N}$, which 
map the space ${\mathfrak{M}}_{k}$ into itself. The standard definition is 
\begin{equation}
(T_{n}f)(\tau) := n^{k-1} \sum_{d | n} d^{-k} \sum_{b=0}^{d-1} 
f \left( \frac{n \tau + bd}{d^{2}} \right),
\end{equation}
\noindent
that in the special case $n = p$ prime, becomes
\begin{equation}
(T_{p}f)(\tau) = p^{k-1} f(p \tau) + \frac{1}{p} 
\sum_{b=0}^{p-1} f \left( \frac{\tau + b}{p} \right). 
\end{equation}
\noindent
In terms of the Fourier expansion of $f \in {\mathfrak{M}}_{k}$, given by
\begin{equation}
f(\tau) = \sum_{m=0}^{\infty} c(m) q^{m},
\end{equation}
\noindent
the action of $T_{n}$ is
\begin{equation}
\left( T_{n}f \right)(\tau) = \sum_{m=0}^{\infty} \gamma_{n}(m) q^{m},
\end{equation}
\noindent
where 
\begin{equation}
\gamma_{n}(m) = \sum_{d | (n,m) } d^{k-1} c \left( \frac{mn}{d^{2}} \right). 
\end{equation}
\noindent 
In particular, when $n = p $ is prime, we have 
\begin{equation}
\gamma_{p}(m) = \begin{cases} 
        c(pm) + p^{k-1}c \left( \frac{m}{p} \right) & \text{ if } p|m,  \\
        c(pm) & \text{ if } p \not | m.
         \end{cases}
\end{equation}

History has shown that the study of Hecke operators is one of the most 
important tools in modern Number Theory, yielding results about the uniform
distributions of points, the eigenvalues of Laplacians on 
various domains, the asymptotic analysis of Fourier coefficients of modular 
forms, and other branches of Number Theory.
\smallskip 

Interesting results were obtained in the last decade when 
the space of modular forms was replaced with the space of rational 
functions (see 
\cite{GilRobins}, \cite{Moll}, and \cite{Ed}).   For example, the spectral 
properties of the operator $U_{n}$ acting on rational functions were 
completely characterized, and corollaries about completely multiplicative 
functions that satisfy linear recurrence sequences were 
obtained (\cite{GilRobins}).

\smallskip 
 For the space $\R$ of rational functions, the coefficients $a_{n}$ in 
(\ref{formal-def}) are the Taylor coefficients of
$A/B \in \R$, with 
$B(x) = 1 + \alpha_{1}x+ 
\cdots + \alpha_{d}x^{d}$ and $A$ a polynomial in $x$, of degree less 
than $d$.   These coefficients $a_n$  are known to satisfy the recurrence 
relation
\begin{equation}
a_{n+d} = - \alpha_{1}a_{n+d-1}- \cdots - \alpha_{d}a_{n}, 
\end{equation}
\noindent
see \cite{Stanley} for details. Thus the 
study of these coefficients employs the theory of linear recurrence sequences 
and their explicit solutions.
One of the main results in \cite{GilRobins} is the complete 
determination of the 
spectrum of $U_{n}$ acting on $\R$:
\begin{equation}
\text{spec}(U_{n}) = \{ \pm n^{k}: \, k \in {\mathbb{N}} \} \cup \{ 0 \}. 
\end{equation}
\noindent
Recent work has produced a description of the corresponding 
rational eigenfunctions, see \cite{Ed, EdBen}.

\medskip 
In this paper, we consider the action of $U_{n}$ on the \emph {set} of 
hypergeometric 
functions
\begin{equation}
\hyp := \left\{ f(x) := \sum_{k=0}^{\infty} c_{k}x^{k}: \, \frac{c_{k+1}}
{c_{k}} \text{ is a rational function of } k \,  \right\}.
\end{equation}
\noindent
We emphasize here that $\hyp$ is a set rather than a vector space, because 
the sum of two hypergeometric functions is \emph{not} in general another 
hypergeometric function.   Nevertheless, this set includes most of 
the classical functions
 as well as all functions of the form
\begin{equation}
\sum_{k=0}^{\infty} R(k)x^{k},
\end{equation}
\noindent
where $R$ is a rational function. 

Every hypergeometric function that we consider has a canonical Taylor 
series representation of the form
\begin{equation}
{_{p}F_{q}}( \mathbf{a}, \mathbf{b}; x) := 
a_{0} 
\sum_{k=0}^{\infty} \frac{(a_{1})_{k} \, (a_{2})_{k} \, \cdots \, (a_{p})_{k} }
{(b_{1})_{k} \, (b_{2})_{k} \, \cdots \, (b_{q})_{k} } \, \frac{x^{k}}{k!},
\end{equation}
\noindent 
where $\mathbf{a} := (a_{1}, \, a_{2}, \cdots, a_{p}) \in {\mathbb{C}}^{p}$ 
and 
$\mathbf{b} := (b_{1}, \, b_{2}, \cdots, b_{q}) \in {\mathbb{C}}^{q}$ are 
the {\em parameters} of ${_{p}F_{q}}$.   These parameters satisfy 
$-b_{i} \not \in \mathbb{N}$.   Here $a_0$ is any complex constant, and  
$(c)_{k} := c(c+1)(c+2) \cdots (c+k-1), \, (c)_{0} := 1$ is the ascending 
factorial symbol.   For example $(1)_k = k!$, and $(0)_k = 0$.

Hypergeometric functions include 
\begin{equation}
f(x) = \sum_{k=0}^{\infty} \frac{x^{k}}{k^{2}+1},
\end{equation}
\noindent
as well as most of the elementary functions. For example, the 
hypergeometric representation of the exponential function is given 
by $e^x = {_1F_1}(a, a; x)$ for any nonzero $a \in \mathbb{C}$.
Similarly, the error function
\[
\text{erf}(x) = \frac{2}{\sqrt{\pi}}\int_{0}^{x}e^{-t^{2}}dt,
\]
can also be represented as a hypergeometric function, namely  
\[
\text{erf}(x) =   \frac{2x}{\sqrt{\pi}}  {_1F_1}(\tfrac{1}{2}, \tfrac{3}{2}; x).
\]
For a more complete discussion of hypergeometric functions, see \cite{Andrews}.

In section \ref{S:hyper} we describe the action of the Hecke operator 
$U_{n}$ on hypergeometric functions. To state the 
results, define
\[
\fjpq :=   \{  x^j    {_{p}F_{q}}(\mathbf{a}, \mathbf{b};x)  :   
\mathbf{a} \in \mathbb{C}^{p}, \, \mathbf{b} \in \mathbb{C}^{q}      \},
\]
for fixed $j \in \N$, and fixed $p,q \in \N$.  This is the set of all 
hypergeometric functions that vanish to order $j$ at the origin,
and have hypergeometric coefficients
\[
 \frac{(a_{1})_{k} \, (a_{2})_{k} \, \cdots \, (a_{p})_{k} }
{(b_{1})_{k} \, (b_{2})_{k} \, \cdots \, (b_{q})_{k} }
\]
with $p$ ascending factorials in the numerator and $q$ ascending factorials 
in the denominator.
Observe that
\begin{align*}
\hyp &:= 
           \{ x^{j} {_{p}F_{q}}(\mathbf{a}, \mathbf{b};x): \,
\text{ for some }j, \, p, \, q \in \mathbb{N},   \text{ and some parameters }
\,  \mathbf{a} \in 
\mathbb{C}^{p}, \, \mathbf{b} \in \mathbb{C}^{q} \, \}  \\
  &= \bigcup_{ j, p,q \in \N  }  \fjpq.
\end{align*}

We establish first the identities
\begin{equation}
U_{n} \left( x^{j} \, {_{p}F_{q}}( \mathbf{a}, \mathbf{b};x \right) 
= x^{j/n} \sum_{k=0}^{\infty} n^{nk(p-q-1)} 
\frac{
(c_{1})_{k} (c_{2})_{k} \cdots (c_{np})_{k} 
}
{
(d_{1})_{k} (d_{2})_{k} \cdots (d_{n(q+1)-1})_{k} 
} \frac{x^{k}}{k!}\in \mathfrak{F}_{(p_{1},q_{1})}^{j_{1}},  \nonumber
\end{equation}
\noindent
when $n$ divides $j$ and
\begin{equation}
U_{n} \left( x^{j} {_{p}F_{q}} ( \mathbf{a}, \mathbf{b}; x \right)  =
x^{1+ \lfloor{ j/n \rfloor}} \sum_{k=0}^{\infty}
n^{nk(p-q-1)} 
\frac{(c_{1})_{k} \cdots (c_{pn})_{k}}
     {(d_{1})_{k} \cdots (d_{(q+1)n-1})_{k}} 
\frac{x^{k}}{k!}\in\mathfrak{F}_{(p_{1},q_{1})}^{j_{2}}. \nonumber
\end{equation}
\noindent
if $n$ does not divide $j$. Here $p_{1}=np, \, q_{1}=n(q+1)-1$ and the new 
parameters $\mathbf{c}, \, \mathbf{d}$ 
are given in (\ref{new}) and (\ref{new1}) respectively. In particular, we
observe in section  \ref{S:hyper} that $U_{n}$ maps $\fjpq$ into itself
if and only if $p=q+1$. These are the {\em balanced} hypergeometric 
functions. Therefore, an eigenfunction of $U_{n}$ must have $p=q+1$. 

\medskip
The eigenfunctions of $U_{n}$ on the space of formal power series are 
described in Section  \ref{S:action} . We consider solutions of 
$U_{n}f = \lambda f$, for $f =x^j  \sum_{n=0}^\infty a_n x^n$  and show that
if $n$ divides $j$, then it follows that $j=0, \, \lambda =1 $ and the 
eigenfunction $f$ must reduce to the rational function $\frac{1}{1-x}$. 
On the other hand, in the case that $n$ 
does not divide $j$, we show that $j$ must be $1$ and the eigenvalue 
$\lambda$ must be of the 
form $n^{a}$, with $a \in \mathbb{Z}$. 

One of the main results here is the complete characterization of the 
spectrum of $U_n$ on hypergeometric functions,
yielding the result that 

\begin{equation}
\text{spec}(U_{n}) = \{ n^{k}: \, k \in {\mathbb{Z}} \} \cup \{ 0 \}. 
\end{equation}

As a corollary, we obtain a number-theoretic characterization of all 
completely multiplicative functions that are also hypergeometric ratios 
of ascending factorials.

\section{A natural inner product on $\hyp$} \label{S:inner}

The set of all hypergeometric functions $\hyp$ can be endowed with a natural inner product
defined by
\begin{equation}
\langle  f, g  \rangle_{R} := \oint_{|z| = R} f(w) \overline{g(w)} \frac{dw}{w}.
\end{equation}
\noindent
For any fixed $R>0$, this inner product defines a function of $R$, and as we shall see shortly it is in fact a real analytic function of $R$.  
Moreover, we shall also see below that $V_n$ is the natural conjugate linear operator to $U_n$ with respect to this inner product.  This fact
is our motivation for introducing the linear operator $V_n$.

The next result describes this inner product in terms of the Taylor series expansions of $f$ and $g$. 

\begin{Lem}
If $f(z) = \sum_{n=0}^{\infty} c_{n}z^{n}$ and 
$g(z) = \sum_{n=0}^{\infty} d_{n}z^{n}$,  then 
\begin{equation}
\langle   f,g \rangle_{R} = 2 \pi i \sum_{n=0}^{\infty} c_{n} \overline{d_{n}} R^{2n}.
\end{equation}
\end{Lem}
\begin{proof}
By definition we have 
\begin{eqnarray}
\langle   f,g \rangle_{R}  & = & \oint_{|z|=R} \sum_{n=0}^{\infty} c_{n}w^{n} \times 
\sum_{m=0}^{\infty} \overline{d_{m}}\overline{w}^{m}  \frac{dw}{w} \nonumber \\
 & = & \oint_{|z|=R} \sum_{n,m =0}^{\infty} c_{n}\overline{d}_{m} w^{n-m-1}  
R^{2m},  \nonumber 
\end{eqnarray}
\noindent
because $\overline{w} = R^{2}/w$ on 
the circle of integration.  Cauchy's integral
formula shows that $n=m$ yielding the result.
\end{proof}

\medskip
This inner product makes sense even for formal power series, although 
we restrict our attention to the set $\hyp$ of
hypergeometric functions.  On $\hyp$, the inner product of any 
two hypergeometric functions is in fact a hypergeometric function
of $R$, as is easily seen by noting that the product of two 
hypergeometric coefficients is another hypergeometric coefficient.

To further develop the algebra of the operators $U_{n}$ and $V_{n}$, we now 
show that $V_{n}$ is the natural conjugate linear operator for $U_{n}$, 
relative to the inner product introduced above.  

\begin{Lem}
Let $f, g \in \hyp$. Then 
\begin{equation}
\langle   U_{n}f,g \rangle_{R} = \langle   f,V_{n}g   \rangle_{R^{n}}.
\end{equation}
\end{Lem}
\begin{proof}
Start with
\begin{eqnarray}
\langle   U_{n}f,g\rangle_{R} & = & \left< \sum_{k=0}^{\infty} c_{kn}z^{k}, 
\sum_{k=0}^{\infty} d_{k}z^{k} \right>_{R} \nonumber \\
& = & \sum_{k=0}^{\infty} c_{nk} \overline{d_{k}} R^{2k}. \nonumber
\end{eqnarray}
\noindent
On the other hand, 
\begin{eqnarray}
\langle f,V_{n}g \rangle_{R} & = & \left< \sum_{k=0}^{\infty} c_{k}z^{k}, 
\sum_{k=0}^{\infty} d_{k}z^{kn} \right> \nonumber \\
& = & \sum_{k=0}^{\infty} c_{k}\overline{h_{k}} R^{2k}, \nonumber 
\end{eqnarray}
\noindent
where we define
\begin{equation}
h_{k} = \begin{cases}
            0 & \text{ if } n \text{ does not divide }k \\
            d_{k/n}  & \text{ if } n \text{ divides }k.
         \end{cases}
\nonumber
\end{equation}
\noindent
It follows that
\begin{equation}
\langle   f,V_{n}g  \rangle_{R} = 
\sum_{k=0}^{\infty} c_{kn}\overline{d_{k}} (R^{n})^{2k} 
= \langle  U_{n}f,g \rangle_{R^{n}}.
\end{equation}
\end{proof}

\medskip

Recall that Hadamard introduced an 
inner product in the space of formal power series 
$\F$, by
\begin{equation}
(f * g)(x) := \sum_{k=0}^{\infty} c_{k}d_{k}x^{k}
\end{equation}
\noindent
This product can be retrieved as a special 
case of $\langle   ,   \rangle_{R}$; namely,
\begin{equation}
\left< \sum_{k=0}^{\infty}c_{k}x^{k}, \overline{\sum_{k=0}^{\infty} d_{k}x^{k} }
\right>_{R^{1/2}} = \sum_{k=0}^{\infty} c_{k}d_{k}R^{k}.
\end{equation}

Thus the Hadamard product is completely equivalent to our inner product.  

\bigskip

\section{Elementary properties of the operators $U_{n}$ and $V_{n}$} 
\label{S:operator} 

In this section we describe elementary properties of the operators defined in 
(\ref{un-def}) and (\ref{vn-def}). 

\begin{Thm}
\label{algebra-1}
Let $m, \, n \in \mathbb{N}$. Then 

\noindent
a) $U_{n} \circ U_{m} = U_{m} \circ U_{n} = U_{nm}$, 

\noindent
b) $V_{n} \circ V_{m} = V_{m} \circ V_{n} = V_{nm}$, 

\noindent
c) $U_{n} \circ V_{n} = \text{Id}$, 

\noindent
d) $U_{n} \circ V_{m} = V_{m/\text{gcd}(m,n)} \circ U_{n/\text{gcd}(m,n)}.$ 
In particular, if $m$ and $n$ are relatively prime, then $U_{n}$ and 
$V_{m}$ commute. 
\end{Thm}
\begin{proof}
Let $f \in \F$ be a formal power series with coefficients $c_{k}$. The first 
two properties follow directly from
\begin{equation}
U_{n} U_{m}f(x) = U_{n} \sum_{k=0}^{\infty} c_{mk}x^{k} = 
\sum_{k=0}^{\infty} c_{nmk}x^{k} = U_{nm}f(x), \nonumber
\end{equation}
\noindent
and similarly for $V_{n} V_{m}$. To establish the third property observe that
\begin{eqnarray}
U_{n}V_{n} f(x) & = & U_{n}V_{n}  
\left( \sum_{k=0}^{\infty} c_{k} x^{k} \right)
  =  U_{n} \left( \sum_{k=0}^{\infty} c_{k} x^{kn} \right) \nonumber \\
 & = & \sum_{k=0}^{\infty} c_{k} x^{k} = f(x). \nonumber 
\end{eqnarray}
\noindent
Finally, 
\begin{equation}
U_{n} \circ V_{m} \sum_{k=0}^{\infty} c_{k}x^{k}  =  
U_{n} \sum_{k=0}^{\infty} c_{k}x^{mk}. \nonumber
\end{equation}
\noindent
To simplify this, let 
\begin{equation}
d_{k} = \begin{cases}
         c_{k/m} & \text{ if } m \text{ divides } k \\
          0  & \text{ if } m \text{ does not divide } k,
\end{cases}
\end{equation}
\noindent
to write 
\begin{equation}
U_{n} \circ V_{m} \left( \sum_{k=0}^{\infty} c_{k}x^{k} \right)   =  
U_{n}  \left( \sum_{k=0}^{\infty} d_{k}x^{k}   \right)
 =  \sum_{k=0}^{\infty} d_{nk}x^{k}. \nonumber
\end{equation}
\noindent
Now observe that $m$ divides $kn$ if and only if $m/\text{gcd}(m,n)$ 
divides $k$. Therefore, with the notation
\begin{equation}
N = \frac{n}{\text{gcd}(m,n)}, \, 
M = \frac{m}{\text{gcd}(m,n)}, 
\nonumber
\end{equation}
\noindent
we have
\begin{equation}
\sum_{k=0}^{\infty} d_{nk}x^{k}  =  
\sum_{i=0}^{\infty}d_{imN}x^{iM}
 =  \sum_{i=0}^{\infty} c_{iN}x^{iM}. \nonumber
\end{equation}
\noindent
Now define $h_{i} = c_{iN}$ to obtain
\begin{equation}
\sum_{i=0}^{\infty} h_{i}x^{iM}  =  V_{M} \left( \sum_{i=0}^{\infty} 
h_{i}x^{i} \right) 
  =  V_{M} \left( \sum_{i=0}^{\infty} c_{iN} x^{i} \right) 
 =  V_{M} \circ U_{N} \left( \sum_{i=0}^{\infty} c_{i}x^{i} \right),
\nonumber
\end{equation}
\noindent
and we have established part d). 
\end{proof}

\medskip

We present now an alternate proof for theorem \ref{algebra-1}.  To do this, we 
must prove an intermediate result.

\medskip
\begin{lemma}{(Associativity of $U$ and $V$)}\\
For all $k,j,m\in \N$, $U_{kj} \circ V_{m} = U_{k}\circ (U_{j}\circ V_{m})$
\end{lemma}

\begin{proof}
Let $f(z) = \sum_{n=0}^{\infty} a_{n}z^{n}$.  For this proof, we will 
evaluate both operators and show they are the same operation.  First, for $U_{kj} \circ V_{m}$, we have that 
\[
(U_{kj}\circ V_{m})f(z) = U_{kj}f(z^{m}) = U_{kj}(\sum_{n\geq 0}a_{n}z^{mn})
\]

Now, let 
\[ 
b_{i}  =
  \begin{cases}
  a_{n}& \text{if } i=mn \text{ for } n \in \N \cup \{ 0 \}\\
  0& \text{otherwise }
  \end{cases} 
\]
so that we can write
\[
U_{kj}(\sum_{n\geq 0}a_{n}z^{mn})=U_{kj}(\sum_{i\geq 0}b_{i}z^{i})
=\sum_{i\geq 0}b_{(kj)i}z^{i}.
\]

Now, for $U_{k}\circ (U_{j}\circ V_{m})$, we have that
\[
U_{k}\circ (U_{j}\circ V_{m})(\sum_{n\geq 0}a_{n}z^{n})=U_{k}(U_{j}(\sum_{n\geq 0}a_{n}z^{mn}))=U_{k}(\sum_{i\geq 0}c_{i}z^{i})
\] 
with  
\[ 
c_{i}  =
  \begin{cases}
  a_{n}& \text{if } ij=mn \text{ for } n\in\N\cup \{ 0 \}\\
  0& \text{otherwise }
  \end{cases}. 
\]

Then we have 

\[
U_{k}(\sum_{i\geq 0}c_{i}z^{i})=\sum_{i\geq 0}c_{ki}z^{i}
\]

We complete the proof by noting that $b_{(kj)i}=c_{ki}$.
\end{proof}

\begin{proof}(Theorem \ref{algebra-1} \emph{alternative})
We can write $U_{n}\circ V_{m} = 
(U_{n/\gcd(m,n)}\circ U_{\gcd(m,n)})\circ (V_{\gcd(m,n)}\circ V_{m/\gcd(m,n)})$ 
by parts (a) and (b) of the theorem.\\
Now, using associativity, we can write
\[
= U_{n/\gcd(m,n)}\circ (U_{\gcd(m,n)}\circ V_{\gcd(m,n)})\circ V_{m/\gcd(m,n)}
\]
By part (c) of the theorem, we have that 
$U_{\gcd(m,n)}\circ V_{\gcd(m,n)} = I$, and we are left with 
$U_{n}\circ V_{m} = U_{m/\gcd(m,n)}\circ V_{\gcd(m,n)}$.  Thus, part 
(d) of the theorem is proven.
\end{proof}

\section{The action of $U_{n}$ on formal power series} 
\label{S:action} 

We now determine an expression for the action of the 
operator $U_{n}$ acting on formal power series 
where we allow the first few coefficients to vanish. This result will be 
employed in our 
study of spectral properties of $U_{n}$ acting on hypergeometric functions. 
For the rest of this section, all functions are assumed to be 
formal power series.

\begin{Thm}
\label{prop1}
Let $j, \, n \in \mathbb{N}$. Then 
\ba
U_{n} \left( x^{j} \sum_{k \geq 0} a_{k}x^{k} \right) & = & 
\begin{cases}
x^{1 + \lfloor{ j/n \rfloor}} \sum a_{n(k+1- \{ j/n \}) } x^{k} & 
\quad \text{ if } n \text{ does not divide } j  \\
 &  \\
x^{\lfloor{ j/n \rfloor}} \sum a_{kn} x^{k} & 
\quad \text{ if } n  \text{ divides }j,
\end{cases}
\nn
\ea
\noindent
where the sums are over $k \geq 0$.
\end{Thm}
\begin{proof}
First observe that 
\ba
U_{n} \left( x^{j} \sum_{k \geq 0} a_{k} x^{k} \right) & = & 
U_{n} \left( \sum_{k \geq 0} a_{k} x^{k+j} \right)  \nn \\
 & = & U_{n} \left( \sum_{k \geq j} a_{k-j} x^{k} \right)  \nn 
\ea
\no
and define 
\ba
\label{def-j}
b_{k} & = & \begin{cases}
   0   & 0 \leq k < j \\
   a_{k-j} & k \geq j 
   \end{cases}
\ea
\no
to write
\begin{equation}
U_{n} \left( x^{j} \sum_{k \geq 0} a_{k}x^{k} \right)  =  
U_{n} \left( \sum_{k \geq 0} b_{k}x^{k} \right)  =
\sum_{k \geq 0} b_{kn}x^{k}. 
\end{equation}

\no
The discussion is divided in two cases according to whether $n$ 
divides $j$  or not. \\

\no
{\bf Case 1}: $n$ does not divide $j$. Then the restriction $kn \geq j$ 
in (\ref{def-j}) is equivalent to 
 $ k  \geq \left\lfloor \frac{j}{n} \right\rfloor + 1$. Thus, 
\ba
\sum_{k \geq 0} b_{kn}x^{k} & = & \sum_{k=0}^{\lfloor{ \frac{j}{n} \rfloor}}
b_{kn}x^{k} + 
\sum_{k=\lfloor{ \frac{j}{n} \rfloor} + 1}^{\infty} b_{kn}x^{k} \nn \\
 & = & \sum_{k=\lfloor{ \frac{j}{n} \rfloor} + 1}^{\infty} a_{kn-j}x^{k}. \nn 
\ea
\no
Now let
$i  = k - \lfloor{ \frac{j}{n} \rfloor} -1$, to obtain
\ba
\sum_{k \geq 0} b_{kn}x^{k} & = & x^{\lfloor{ \frac{j}{n} \rfloor} +1} 
\sum_{i \geq 0} a_{ni + n \lfloor{ \frac{j}{n} \rfloor} + n-j} x^{i}. 
\ea
\no
Now use $\frac{j}{n}  =  \lfloor{ \frac{j}{n} \rfloor} + \left\{ \frac{j}{n} \right\}$
to obtain
\ba
U_{n} \left( x^{j} \sum_{k \geq 0} a_{k}x^{k} \right) & = & 
x^{\lfloor{ \frac{j}{n} \rfloor} +1} \sum_{i \geq 0} a_{n( i + 1 -  
\{ \frac{j}{n} \} } x^{i}. 
\ea
\noindent
This is the result when $n$ does not divide $j$. 

\medskip

\no
{\bf Case 2}: $n$ divides $j$. Then $kn \geq j$ is now equivalent to
$k \geq \frac{j}{n} = \lfloor{ \frac{j}{n} \rfloor} $ and 
\ba
\sum_{k \geq 0} b_{kn}x^{k} & = & 
\sum_{k \geq  \lfloor{ j/n \rfloor}} b_{kn}x^{k} \nn \\
& = & \sum_{k \geq \lfloor{ j/n \rfloor}} a_{kn-j}x^{k} \nn \\
& = & x^{\lfloor{j/n \rfloor}} \sum_{i \geq 0} a_{n i }x^{i} \nn 
\ea
\no
so that
\ba
U_{n} \left( x^{j} \sum_{k \geq 0} a_{k}x^{k} \right) & = & 
x^{\lfloor{ \frac{j}{n} \rfloor} } \sum_{\nu \geq 0} a_{n \nu } x^{\nu}.
\ea
\noindent
This concludes the proof.
\end{proof}

The previous expressions for $U_{n}$ are now used to derive some 
elementary properties of its eigenfunctions on the space of formal power 
series. 

\begin{Prop}\label{gen-eig}
Assume $U_{n}$ has an eigenfunction of the form 
\begin{equation}
f(x) = x^{j} \sum_{k=0}^{\infty} a_{k}x^{k}, 
\end{equation}
\noindent 
with eigenvalue $\lambda$. If $n$ divides $j$, then we conclude that 
$j=0$ and $\lambda = 1$.  If $n$ does not 
divide $j$, then we conclude that  $j = 1$. 
\end{Prop}
\begin{proof}
Assume $n$ divides $j$ and match the leading order terms of 
$f$ and $U_{n}f$. Theorem \ref{prop1} shows that $x^{j} = x^{j/n}$ yielding 
$j=0$. Now compare the constnat term in the eigenvalue equation to get 
$\lambda = 1$. In the case $n$ does not divide $j$ the same comparison yields 
\begin{equation}
j = 1 + \lfloor{  j/n \rfloor}.
\end{equation}
\noindent
This implies $j=1$. Indeed, let $j = \alpha n + \beta$ with 
$0 < \beta < n$. Then we have $j = 1 + \alpha$, 
and this yields 
\begin{equation}
1 - \beta = \alpha(n-1). 
\label{eqsign}
\end{equation}
\noindent
It follows that $\beta = 1$ and $\alpha = 0$, otherwise both sides of 
(\ref{eqsign}) have different signs. We conclude that $j = 1 + \alpha = 1$. 
\end{proof}
\bigskip

Hence even for a formal power series $f$, we see that the assumption that 
$f$  is an eigenfunction of the Hecke operator $U_n$ imposes the restriction
that $f$ can only vanish to order zero or one.   

For the sake of completeness, we describe the trivial eigenfunctions 
of the composition of operators $U_{n} \circ V_{n}$ and 
$V_{n} \circ U_{n}$. Theorem \ref{algebra-1} shows that
$U_{n} \circ V_{n}$ is the identity. The next result describes the composition
$V_{n} \circ U_{n}$.

\begin{Thm}
The only eigenvalue of $V_{n} \circ U_{n}$ is $\lambda = 1$. Moreover, given
any formal power series  $f(x) = \sum_{k=0}^{\infty} b_{k}x^{k}$, the 
function $g(x) = \sum_{k=0}^{\infty} a_{k}x^{k}$, with
\begin{equation}
a_{k} = \begin{cases}
         b_{k} & \text{ if } n \text{ divides }k \\
         0  & \text{ if } n \text{ does not divide }k
\end{cases}
\end{equation}
\noindent
is an eigenfunction of $V_{n} \circ U_{n}$, with eigenvalue $1$.
\end{Thm}
\begin{proof}
The result follows directly from the identity
\begin{equation}
\left(    V_{n} \circ U_{n}    \right)   
\left( \sum_{k=0}^{\infty} a_{k}x^{k} \right) = 
\sum_{k=0}^{\infty} a_{kn}x^{k}.
\end{equation}
\end{proof}

\section{The hypergeometric functions} \label{S:hyper} 

In this section we use Theorem \ref{prop1} to describe the 
action of $U_{n}$ on the set $\hyp$ of all hypergeometric functions.
We recall that a hypergeometric function is defined by
\begin{equation}
{_{p}F_{q}}( \mathbf{a}, \mathbf{b}; x) := 
\sum_{k=0}^{\infty} \frac{(a_{1})_{k} \, (a_{2})_{k} \, \cdots \, (a_{p})_{k} }
{(b_{1})_{k} \, (b_{2})_{k} \, \cdots \, (b_{q})_{k} } \, \frac{x^{k}}{k!},
\end{equation}
\noindent 
where $\mathbf{a} := (a_{1}, \, a_{2}, \cdots, a_{p})$ and 
$\mathbf{b} := (b_{1}, \, b_{2}, \cdots, b_{q})$ are the parameters 
of $_{p}F_{q}$. These parameters are non-zero complex numbers.   We 
begin by stating explicitely the action of $U_{n}$ on $\fjpq$ as 
the main result of this section.
\medskip

\begin{Thm}
\label{thm-action}
Let $j, \, n \in \mathbb{N}$. The action of $U_{n}$ on the 
class $\fjpq$, that is, on functions of the form 
\begin{equation}
f_{p,q,j} = x^{j} {_{p}F_{q}}( \mathbf{a}, \mathbf{b};x) 
\end{equation}
\noindent
is characterized as follows.  

\medskip
\noindent
If $n$ divides $j$, we have
\begin{equation}
U_{n} \left( x^{j} \sum_{k=0}^{\infty} 
\frac{(a_{1})_{k} \cdots (a_{p})_{k}}
     {(b_{1})_{k} \cdots (b_{q+1})_{k}} x^{k} \right) = 
x^{j/n} \sum_{k=0}^{\infty} 
\frac{
(c_{1})_{k} (c_{2})_{k} \cdots (c_{np})_{k} 
}
{
(d_{1})_{k} (d_{2})_{k} \cdots (d_{n(q+1)-1})_{k} 
} \frac{x_{1}^{k}}{k!}\in \mathfrak{F}_{(p_{1},q_{1})}^{j_{1}},  \nonumber
\end{equation}
where we define  the parameters
\begin{eqnarray}
c_{in+l} & = & \frac{a_{i+1}+l-1}{n}, \quad \text{ for } 0 \leq i \leq p-1, \, 
1 \leq l \leq n  \label{new} \\
d_{in+l} & = & \frac{b_{i+1}+l-1}{n}, \quad \text{ for } 0 \leq i \leq q, \, 
1 \leq l \leq n,  \nonumber 
\end{eqnarray} \\
\noindent
and $j_{1} = j/n$. The new variable is $x_{1} = n^{n(p-q-1)}x$. 

\medskip

\noindent
If $n$ does not divide $j$, we have

\begin{equation}
U_{n} \left( x^{j} \sum_{k=0}^{\infty} 
\frac{(a_{1})_{k} \cdots (a_{p})_{k}}
     {(b_{1})_{k} \cdots (b_{q+1})_{k}} x^{k} \right) = 
x^{1+ \lfloor{ j/n \rfloor}} \sum_{k=0}^{\infty} 
\frac{(c_{1})_{k} \cdots (c_{pn})_{k}}
     {(d_{1})_{k} \cdots (d_{(q+1)n-1})_{k}} 
\frac{x_{1}^{k}}{k!}\in \mathfrak{F}_{(p_{2},q_{2})}^{j_{2}}, \nonumber
\end{equation}
where we now define the parameters 

\begin{eqnarray}
c_{in+l} & = & \frac{a_{i+1}+r+l}{n}, \quad \text{ for } 0 \leq i \leq p-1, \, 
1 \leq l \leq n,  \label{new1} \\
d_{in+l} & = & \frac{b_{i+1}+r+l}{n}, \quad \text{ for } 0 \leq i \leq q, \, 
1 \leq l \leq n,  \nonumber 
\end{eqnarray}
\noindent
and $j_{2} = 1 + \lfloor{j/n \rfloor}$. The new variable $x_{1}$ is defined as 
above. 
\end{Thm}
\bigskip

Before proving this theorem, we first need to state some intermediate 
results. The next lemma allows for a simplication of the ascending 
factorial function on an arithmetic progression of indices.

\begin{Lem}
\label{mult-poch}
Let $k, \, n \in \mathbb{N}$ and $a \in \mathbb{R}$. Then
\begin{equation}
(a)_{kn}  =  n^{kn} \, \prod_{j=0}^{n-1} \left( \frac{a+j}{n} \right)_{k}.
\nn
\end{equation}
\end{Lem}
\begin{proof}
Start with 
\begin{equation}
(a)_{kn}  =  \prod_{i=0}^{kn -1} (a + i) 
          =  n^{kn} \, \prod_{i=0}^{kn-1} \left( \frac{a}{n} + \frac{i}{n} 
\right), \nn
\end{equation}
\no
and then collect terms according to classes modulo $n$.
\end{proof}

\medskip

In order to evaluate 
\begin{equation}
U_{n} \left( x^{j} {_{p}F_{q}}(\mathbf{a}, \mathbf{b} ) \right) = 
U_{n} \left( x^{j} \sum_{k=0}^{\infty} \frac{(a_{1})_{k} (a_{2})_{k} 
\cdots (a_{p})_{k}} {(b_{1})_{k} (b_{2})_{k} \cdots (b_{q})_{k} (b_{q+1})_{k}}
x^{k} \right),
\nonumber 
\end{equation}
where we have used $k! = (1)_{k}$ and defined $b_{q+1} = 1$, we observe that 
by Theorem \ref{prop1}  
the discussion should be divided into two cases according to whether or 
not $n$ divides $j$.\\

\noindent
{\bf Case 1}: $n$ divides $j$. Theorem \ref{prop1} yields
\begin{equation}
U_{n} \left( x^{j} {_{p}F_{q}}(\mathbf{a}, \mathbf{b} ) \right) = 
x^{j/n} \sum_{k=0}^{\infty} 
\frac{(a_{1})_{kn} \cdots (a_{p})_{kn} }
{(b_{1})_{kn} \cdots (b_{q+1})_{kn} }  x^{k}, 
\end{equation}
\noindent 
and using Lemma \ref{mult-poch} we can  write this as 
\begin{equation}
U_{n} \left( x^{j} {_{p}F_{q}}(\mathbf{a}, \mathbf{b} ) \right) = 
x^{j/n} \sum_{k=0}^{\infty} n^{(p-q-1)kn} 
\left( 
\prod_{j=1}^{p} \prod_{i=0}^{n-1} \left( \frac{a_{j} + i}{n} \right)_{k} 
\right)
\times 
\left( 
\prod_{j=1}^{q+1} \prod_{i=0}^{n-1} \left( \frac{b_{j} + i}{n} \right)_{k} 
\right)^{-1} x^{k}.
\nonumber
\end{equation}

\noindent 
Now recall that $b_{q+1}=1$, so the $d$ parameters corresponding to 
$i=q$ are 
$1/n, \, 2/n, \cdots, (n-1)/n, \, 1$. The total number of $d$-parameters 
is now reduced by $1$, in order to write the result in the canonical 
hypergeometric form: 
\begin{equation}
U_{n} \left( x^{j} {_{p}F_{q}}(\mathbf{a}, \mathbf{b} ) \right) = 
x^{j/n} \sum_{k=0}^{\infty} n^{(p-q-1)kn} 
\frac{
(c_{1})_{k} (c_{2})_{k} \cdots (c_{np})_{k} 
}
{
(d_{1})_{k} (d_{2})_{k} \cdots (d_{n(q+1)-1})_{k} 
} \frac{x^{k}}{k!}.  \nonumber
\end{equation}
\medskip

\begin{Lem}
The parameters $\mathbf{a}, \mathbf{b}, \mathbf{c}$ and $\mathbf{d}$ satisfy
\begin{equation}
\sum_{i=1}^{np} c_{i} - \sum_{i=1}^{n(q+1)-1} d_{i}  = 
\sum_{i=1}^{p} a_{i} - \sum_{i=1}^{q} b_{i} + \frac{(n-1)}{2} (p-q-1). 
\nn
\end{equation}
\end{Lem}
\begin{proof}
The new parameters are 
\ba
\frac{a_{1}}{n}, \, \frac{a_{1}+1}{n}, \cdots, \frac{a_{1}+n-1}{n}, \,
\frac{a_{p}}{n}, \cdots, \frac{a_{p}+n-1}{n} \nn
\ea
\no
and 
\ba
\frac{b_{1}}{n}, \, \frac{b_{1}+1}{n}, \cdots, \frac{b_{1}+n-1}{n}, \, 
\frac{b_{q+1}}{n} = \frac{1}{n}, \cdots, \frac{b_{q+1}+n-2}{n} = \frac{n-1}{n}
\nn
\ea
\no
and the identity is now easy to check. 
\end{proof}

\medskip

\begin{Cor}
Suppose $p = q+1$, then $U_{n}$ preserves the quantity $\sum a_{i} - 
\sum b_{i}$.
\end{Cor}

\medskip

\noindent
{\bf Case 2}: $n$ does not divide $j$. Proposition \ref{prop1} now gives 
\begin{equation}
U_{n} \left( x^{j} \sum_{k=0}^{\infty} 
\frac{(a_{1})_{k} \cdots (a_{p})_{k} }
     {(b_{1})_{k} \cdots (b_{q+1})_{k} } x^k   \right)  
= x^{1+ \lfloor{j/n \rfloor}}
\sum_{k=0}^{\infty} 
\frac{(a_{1})_{N} \cdots
(a_{p})_{N}  }
{ (b_{1})_{N}  \cdots 
(b_{q+1})_{N}  }x^k,
\end{equation}
\noindent
where $b_{q+1}  = 1$ and we define $N = n(k+1 - \{ j/n \} )$. Observe 
that $0 < \left\{ \frac{j}{n} \right \} < 1$, 
thus  $nk < N < n(k+1)$. The next result 
simplifies the Pochhammer symbols.

\begin{Lem}
Let $a \in \mathbb{C}$ and $j, \, n \in \mathbb{N}$ with $j$ not divisible 
by $n$. Define $N = n(k+1- \{ j/n \} )$ and $r = n(1 - \{ j/n \}) -1$. Then
\begin{equation}
(a)_{N} = n^{nk} (a)_{r+1} \prod_{i=r+1}^{r+n} \left( \frac{a+i}{n} \right)_{k} 
\end{equation}
\end{Lem}
\begin{proof}
Start with 
\begin{eqnarray}
(a)_{N} & = & a(a+1)(a+2) \cdots (a+N-1) \nonumber \\
 & = & n^{N} \left[ \frac{a}{n} \left( \frac{a}{n} + \frac{1}{n} \right) 
\cdots 
 \left( \frac{a}{n} + \frac{N-1}{n} \right) \right], \nonumber 
\end{eqnarray}
\noindent 
and now grouping terms modulo $n$ as follows:
\begin{eqnarray}
n^{-N} (a)_{N} & = & \left( \frac{a}{n} \right) 
\cdot \left( \frac{a}{n} + 1 \right) \cdots 
\left( \frac{a}{n} + k-1 \right) \nonumber \\ 
 & \times  & \left( \frac{a}{n} + \frac{1}{n} \right) 
\cdot \left( \frac{a}{n} + \frac{1}{n} + 1 \right) \cdots 
\left( \frac{a}{n} + \frac{1}{n} + k-1 \right) \nonumber \\ 
 & & \cdots \nonumber \\
 & \times  & \left( \frac{a}{n} + \frac{n-1}{n} \right) 
\cdot \left( \frac{a}{n} + \frac{n-1}{n} + 1 \right) \cdots 
\left( \frac{a}{n} + \frac{n-1}{n} + k-1 \right) \nonumber 
\end{eqnarray}
\noindent 
multiplied by the factor 
\begin{equation}
\left( \frac{a}{n} + k \right) \cdot 
\left( \frac{a}{n} +  \frac{1}{n} + k \right) \cdots
\left( \frac{a}{n} +  \frac{r}{n} + k \right), \nonumber
\end{equation}
\noindent 
that appears because $n$ does not divide $j$. Therefore we have
\begin{equation}
(a)_{N} = n^{N} \prod_{i=0}^{n-1} \left( \frac{a+i}{n} \right)_{k} \times
\prod_{i=0}^{r} \left( \frac{a+i}{n} + k \right),
\end{equation}
\noindent
where the second product is {\em not} the Pochhammer symbol. Now employ the 
relation
\begin{equation}
k+ c = c \frac{(c+1)_{k}}{(c)_{k}},
\end{equation}
\noindent
to write
\begin{equation}
(a)_{N} = n^{N} \prod_{i=0}^{n-1} \left( \frac{a+i}{n} \right)_{k} \,
\prod_{i=0}^{r} \left( \frac{a+i}{n} \right) \, \cdot  \, 
\prod_{i=0}^{r} \left( \frac{a+i}{n} +1 \right)_{k} \Big{/}
\prod_{i=0}^{r} \left( \frac{a+i}{n} \right)_{k}. \nonumber
\end{equation}
\noindent
This expression reduces to the stated formula.
\end{proof}

\medskip

The transformation above yields 
\begin{equation}
U_{n} \left( x^{j} \sum_{k=0}^{\infty} 
\frac{(a_{1})_{k} \cdots (a_{p})_{k}}
     {(b_{1})_{k} \cdots (b_{q+1})_{k}} x^{k} \right) = 
x^{1+ \lfloor{ j/n \rfloor}} \sum_{k=0}^{\infty} 
n^{nk(p-q-1)} 
\frac{(c_{1})_{k} \cdots (c_{pn})_{k}}
     {(d_{1})_{k} \cdots (d_{(q+1)n})_{k}} x^{k}. \nonumber
\end{equation}

The special case $p = q+1$ provides a simpler situation, in which 
the coefficient
$n^{  nk(p-q-1) }$  does not appear in the resulting series.  \\

\bigskip
\begin{Thm}
\label{pqplus1}
Let $j, \, n \in \mathbb{N}$ and assume $p = q+1$. 

\medskip
If $n$ divides $j$, we have 
\begin{equation}
U_{n} \left( x^{j} {_{p}F_{q}} \left( \mathbf{a}, \mathbf{b}; x \right) 
\right) = 
x^{j/n} {_{np}F_{np-1}} \left( \mathbf{c}, \mathbf{d}; x \right),
\end{equation}
\noindent
where $\mathbf{c}, \, \mathbf{d}$ are defined in (\ref{new}). 

\medskip
If $n$ 
does not divide $j$, we have 
\begin{equation}
U_{n} \left( x^{j} {_{p}F_{q}} \left( \mathbf{a}, \mathbf{b}; x \right)  
\right) = 
x^{1+ \lfloor{j/n \rfloor}} {_{np}F_{np-1}} 
\left( \mathbf{c}, \mathbf{d}; x \right),
\end{equation}
\noindent
where $\mathbf{c}, \, \mathbf{d}$ are defined in (\ref{new1}).
\end{Thm}
 \bigskip

\section{The eigenvalue equation}
\label{S:eigen}

In this section  we focus on the  spectral properties of the operator
$U_{n}$, as the spectral  properties of the operators $V_{n}$ are trivial. It 
is here that we encounter more subtle ideas. We
describe first the eigenfunctions of the operator $U_{n}$ of the 
form $x^{j} {_{p}F_{q}}(\mathbf{a}.
\mathbf{b}; x)$.  
\noindent
That is, we look for parameters $p, \, q \in \mathbb{N}$ and 
complex numbers 
\begin{equation}
a_{1}, \, a_{2}, \cdots, a_{p}; \, b_{1}, \, b_{2}, \, 
\cdots, b_{q}
\end{equation}
\noindent
such that, with $\mathbf{a} = (a_{1}, \cdots, a_{p})$ and  
$\mathbf{b} = (b_{1}, \cdots, b_{q}),$ we have
\begin{equation}
\label{eigen1}
U_{n} \left( x^{j} {_{p}F_{q}}( \mathbf{a}, \mathbf{b}; x) \right) 
= \lambda x^{j} {_{p}F_{q}} ( \mathbf{a}, \mathbf{b};x).
\end{equation}

The results from Theorem \ref{thm-action} showed that the action of $U_{n}$ 
on $x^{j}_{p}F_{q}$ depends on whether or not $n$ divides $j$, which by 
Theorem \ref{gen-eig} reduces to the cases $j=0$ and $j=1$ when 
considering eigenfunctions of $U_{n}$. \\

\noindent
{\bf Case 1}: $j=0$. Under this condition we show that the 
eigenfunction reduces to a rational function. \\

\begin{Lem}
Assume $n$ divides $j$ and that (\ref{eigen1}) has a nontrivial solution.   
Then, for all $k \in \mathbb{N}$, we have 
\begin{equation}
\prod_{j=1}^{p} \prod_{i=0}^{n-1} (a_{j}+nk+i) \times 
\prod_{j=1}^{q+1} (b_{j}+k) = 
\prod_{j=1}^{p} (a_{j}+j ) \times 
\prod_{j=1}^{q+1} \prod_{i=0}^{n-1} (b_{j}+nk+i). 
\end{equation}
\end{Lem}
\begin{proof}
Comparing terms of the equation $U_{n}f = f$ yields
\begin{equation}
\frac{(a_{1})_{nk} \, (a_{2})_{nk} \cdots (a_{p})_{nk}}
{(b_{1})_{nk} \, (b_{2})_{nk} \cdots (b_{q+1})_{nk} } = 
\frac{(a_{1})_{k} \, (a_{2})_{k} \cdots (a_{p})_{k}}
{(b_{1})_{k} \, (b_{2})_{k} \cdots (b_{q+1})_{k} }. 
\label{poly}
\end{equation}
\noindent 
Replace $k$ by $k+1$,  divide the two equations  and use 
\begin{equation}
\frac{(a)_{k+1}}{(a)_{k}} = a+k, \text{ and } 
\frac{(a)_{n(k+1)}}{(a)_{k}} = (a+nk)(a+nk+1) \cdots (a+nk+n-1)
\end{equation}
\noindent
to produce the result. 
\end{proof}

\begin{Lem}
Assume $n$ divides $j$ and that (\ref{eigen1}) has a nontrivial solution.   
Then $p = q+1$.
\end{Lem}
\begin{proof}
Comparing the degrees of the left and right hand side of (\ref{poly}) gives 
$pn + q+1  = p+n(q+1)$. 
\end{proof}

\begin{Prop}
\label{sets-equal}
Assume $n$ divides $j$ and that (\ref{eigen1}) has a nontrivial solution.  Then 
\begin{equation}
\left\{ a_{i}, \frac{b_{i}}{n}, 
\frac{b_{i}+1}{n}, \cdots, \frac{b_{i}+n-1}{n} 
\,  \right\}_{i=1}^{p}  = 
\left\{ b_{i}, \frac{a_{i}}{n}, 
\frac{a_{i}+1}{n}, \cdots, \frac{a_{i}+n-1}{n} 
\,  \right\}_{i=1}^{p}.
\nonumber
\end{equation}
\end{Prop}
\begin{proof}
The roots of the left and right hand side of (\ref{poly}) must match. 
\end{proof}

We now show that the results of this Proposition imply that the 
parameters must match: $a_{i} = b_{i}$ for all indices.

\begin{Prop}
Assume $n$ divides $j$ and that (\ref{eigen1}) has a nontrivial solution.  
Then, for any $k \in \mathbb{N}$, we have 
\begin{equation}
\sum_{i=1}^{p} a_{i}^{k} = \sum_{i=1}^{p} b_{i}^{k}. 
\end{equation}
\end{Prop}
\begin{proof}
Proof by induction on $k$. The case $k=1$ comes from matching the coefficients
of the next to leading order in $k$. Indeed, this matching yields
\begin{equation}
\sum_{i=1}^{p} a_{i} + \sum_{j=1}^{n-1} \sum_{i=1}^{p} ( b_{i} + j) = 
\sum_{i=1}^{p} b_{i} + \sum_{j=1}^{n-1} \sum_{i=1}^{p} ( a_{i} + j),
\nonumber 
\end{equation}
\noindent
and the case $k=1$ holds. In order to check it for
$k=2$, add the squares of the elements in Lemma \ref{sets-equal} to obtain, 
from the left hand side the expression
\begin{equation}
\sum_{i=1}^{p} a_{i}^{2} + \frac{1}{n^{2}} 
\sum_{i=1}^{p} \sum_{j=0}^{n-1} (b_{i}^{2} + 2j b_{i} + j^{2} ) = 
\sum_{i=1}^{p} a_{i}^{2} + \frac{1}{n} \sum_{i=1}^{p} b_{i}^{2} + 
\frac{2}{n^{2}} \sum_{i=1}^{p} b_{i} \times \sum_{j=0}^{n-1} j +
\frac{p}{n^{2}} \sum_{j=0}^{n-1} j^{2}. \nonumber 
\end{equation}
\noindent
Matching with the corresponding expression from the right hand side and using 
the statement for  $k=1$, yields 
\begin{equation}
\sum_{i=1}^{p} a_{i}^{2}  = 
\sum_{i=1}^{p} b_{i}^{2}. 
\end{equation}
\noindent
The higher moments can be established along these lines.
\end{proof}

\begin{Prop}
\label{moments}
Assume two sets $\{ a_{j} : \, 1 \leq j \leq n \}$ and 
$\{ b_{j} : \, 1 \leq j \leq n \}$  of complex numbers satisfy 
\begin{equation}
\sum_{i=1}^{p} a_{i}^{k}  = 
\sum_{i=1}^{p} b_{i}^{k},
\end{equation}
\noindent
for every $k \in \mathbb{N}$. Then, after a possible rearrangement of terms of one of these sets, we have  $a_{i} = b_{i}$ for all $i$. 
\end{Prop}
\begin{proof}
For $\mathbf{a},\mathbf{b}\in\CC^{p}$ and $N\in\N$, let 

\begin{equation}
\mu_{N}(\mathbf{a})=\sum_{i=1}^{p}a_{i}^{N}
\end{equation}

 and let 
 
\[f_{\mathbf{a}}(t)=\sum_{j=0}^{\infty}\mu_{j}(\mathbf{a})\frac{t^{j}}{j!} \text{,   } f_{\mathbf{b}}(t)=\sum_{j=0}^{\infty}\mu_{j}(\mathbf{b})\frac{t^{j}}{j!}
\] 

be the generating functions of $\mu_{N}(\mathbf{a})$ and $\mu_{N}(\mathbf{b})$, respectively.  Now assume  $f_{\mathbf{a}}=f_{\mathbf{b}}$, i.e.

\[\sum_{j=0}^{\infty}\mu_{j}(\mathbf{a})\frac{t^{j}}{j!}=\sum_{j=0}^{\infty}\mu_{j}(\mathbf{b})\frac{t^{j}}{j!}
\]

Expanding further gives

\[
\sum_{j=0}^{\infty}\sum_{i=1}^{p}a_{i}^{j}\frac{t^{j}}{j!}=\sum_{j=0}^{\infty}\sum_{i=1}^{p}b_{i}^{j}\frac{t^{j}}{j!}
\]

Since $\mu_{j}$ is defined as a sum of finite terms, we can change the order of summation.

 \[
 \sum_{i=1}^{p}\sum_{j=0}^{\infty}a_{i}^{j}\frac{t^{j}}{j!}=\sum_{i=1}^{p}\sum_{j=0}^{\infty}b_{i}^{j}\frac{t^{j}}{j!}
 \]  
 
 This yields 
\begin{equation}
e^{a_{1}t} + e^{a_{2}t} + \cdots + e^{a_{p}t} = 
e^{b_{1}t} + e^{b_{2}t} + \cdots + e^{b_{p}t}. 
\label{exp-match}
\end{equation}
\noindent
Suppose first that $a_{i}, \, b_{i} \in \mathbb{R}$ and order them as 
\[
a_{1} \leq a_{2} \leq \cdots \leq  a_{p} \text{ and }
b_{1} \leq b_{2} \leq \cdots \leq  b_{p}. 
\]
\noindent
Eliminate from (\ref{exp-match}) all the terms for which the  $a's$ and $b's$
match,  to assume that $a_{1} < b_{1}$. Then
\begin{equation}
1 + e^{(a_{2}-a_{1})t} + \cdots + e^{(a_{p}-a_{1})t} = 
e^{(b_{1}-a_{1})t} + e^{(b_{2}-a_{1})t} + \cdots + e^{(b_{p}-a_{1})t}. 
\nonumber
\end{equation}
\noindent
Finally, let $t \to -\infty$ to get a  contradiction.
\end{proof}
\medskip
We summarize the previous discussion in the following Theorem.

\begin{Thm}
\label{eigenfunction-1}
Suppose there exists an eigenfunction of  $U_{n}$ of the form
\begin{equation}
f(x)=   \sum_{k=0}^{\infty} 
\frac{ (a_{1})_{k} (a_{2})_{k} \cdots (a_{p})_{k} }
{(b_{1})_{k} (b_{2})_{k} \cdots (b_{q+1})_{k} }x^{k}
\end{equation}
\noindent 
corresponding to the eigenvalue $\lambda$. Then $\lambda = 1, \, 
p = q+1$ and $a_{i} = b_{i}$ for all $i$. Therefore $f(x) = \frac{1}{1-x}$.
\end{Thm}

\medskip

\noindent
{\bf Case 2}:    $j=1$. Under this condition we show that the 
spectrum of the operator $U_{n}$  is the  set $\{ n^{i}: \, i \in \mathbb{Z} 
\}$.   The corresponding eigenfunctions are the polylogarithm functions
\begin{equation}
\text{PolyLog}_{i}(x) := \sum_{k=1}^{\infty} k^{i}x^{k}, 
\end{equation}
\noindent
corresponding to the eigenvalue $n^{i}$ with negative $i$,  
and the eigenfunctions 
\[
 \left(   x \frac{d}{dx}\right)^i  \left(    \frac{1 } { 1-x}  \right)
\]
 corresponding to the eigenvalue $n^i$ with non-negative $i$.

\begin{Example}
\label{dilog}
The dilogarithm function 
\begin{equation}
\text{Li}_{2}(x) := \sum_{k=1}^{\infty} \frac{x^{k}}{k^{2}}
\end{equation}
\noindent
satisfies
\begin{equation}
U_{n} \left( \text{Li}_{2}(x)  \right)= \frac{1}{n^{2}} \text{Li}_{2}(x).
\end{equation}
\noindent
Therefore $1/n^{2} \in \spun$.  The dilogarithm function admits the 
hypergeometric representation
\begin{equation}
\text{Li}_{2}(x) = x \, {_{3}F_{2}} \left( 
\begin{pmatrix}1 &   & 1  &  & 1 \\   & 2 &  & 2 & \end{pmatrix};  x \right)
\end{equation}
\end{Example}

\medskip

We now explore properties of eigenfunctions of the operator $U_{n}$. \\

\begin{Prop}
\label{big-poly}
Assume (\ref{eigen1}) has a nontrivial solution with $j=1$. 
Then, for all $k \in \mathbb{N}$, we have 
\begin{equation}
\prod_{j=1}^{p} (a_{j}+k-1) \cdot 
\prod_{j=1}^{q+1} \prod_{i=-1}^{n-2} (b_{j}+nk+i)  = 
\prod_{j=1}^{q+1} (b_{j}+k-1) \cdot 
\prod_{j=1}^{p} \prod_{i=-1}^{n-2} (a_{j}+nk+i). 
\end{equation}
\end{Prop}
\begin{proof}
Assume the eigenfunction has the form 
\begin{equation}
f(x) = x \, \sum_{k=0}^{\infty} c_{k}x^{k}. 
\end{equation}
\noindent
Comparing coefficients in the equation $U_{n}f = \lambda f$ gives 
\begin{equation}
c_{nk-1} = \lambda c_{k-1}. 
\end{equation}
\noindent
Replacing the standard hypergeometric type yields
\begin{equation}
\frac{(a_{1})_{nk-1} \cdots (a_{p})_{nk-1}}
{(b_{1})_{nk-1} \cdots (b_{q+1})_{nk-1}} = \lambda
\frac{(a_{1})_{k-1} \cdots (a_{p})_{k-1}}
{(b_{1})_{k-1} \cdots (b_{q+1})_{k-1}}. 
\label{crucial}
\end{equation}
\noindent
Replace $k$ by $k+1$ and divide the two corresponding equations to obtain 
the result. 
\end{proof}

\begin{Lem}
\label{lemma}
Assume $j=1$ and that (\ref{eigen1}) has a nontrivial 
solution of the form 
\begin{equation}
f(x) = x \sum_{k=0}^{\infty} \frac{(a_{1})_{k} \cdots (a_{p})_{k} }
{(b_{1})_{k} \cdots (b_{q+1})_{k} }x^{k}. 
\end{equation}
\noindent
Then $p = q+1$. 
\end{Lem}
\begin{proof}
Compare the degrees on both sides of the polynomial in Lemma \ref{big-poly}.
\end{proof}

\begin{Prop}
\label{prop-roots}
Assume 
\begin{equation}
f(x) = x \sum_{k=0}^{\infty} \frac{(a_{1})_{k} \cdots (a_{p})_{k} }
{(b_{1})_{k} \cdots (b_{p})_{k} }x^{k}
\end{equation}
\noindent 
is an eigenfunction for $U_{n}$. Then 
\begin{equation}
\left\{ a_{i}-1, \frac{b_{i}-1}{n}, 
\frac{b_{i}}{n}, \cdots, \frac{b_{i}+n-2}{n} 
\,  \right\}_{i=1}^{p}  = 
\left\{ b_{i}-1, \frac{a_{i}-1}{n}, 
\frac{a_{i}}{n}, \cdots, \frac{a_{i}+n-2}{n} 
\right\}_{i=1}^{p}. 
\nonumber
\end{equation}
\end{Prop}
\begin{proof}
These  are the roots of both sides of the polynomial in Lemma \ref{big-poly}.
\end{proof}

We now show that this equality of sets imposes severe restrictions on the 
eigenvalues and eigenfunctions of the operator $U_{n}$. We discuss first the 
eigenvalues. 

\begin{Prop}
Assume 
\begin{equation}
f(x) = x \sum_{k=0}^{\infty} \frac{(a_{1})_{k} \cdots (a_{p})_{k} }
{(b_{1})_{k} \cdots (b_{p})_{k} }x^{k}
\end{equation}
\noindent 
is an eigenfunction for $U_{n}$. Let 
\begin{equation}
\gamma_{a} := \left| \{ i \in \{ 1, \, 2, \cdots, p \}: 
\text{ such that } a_{i} =1 \} \right|,
\end{equation}
\noindent
and similarly for $\gamma_{b}$. Then 
\begin{equation}
n^{\gamma_{a}} (a_{1})_{n-1} (a_{2})_{n-1} \cdots (a_{p})_{n-1}  = 
n^{\gamma_{b}} (b_{1})_{n-1} (b_{2})_{n-1} \cdots (b_{p})_{n-1}. 
\label{interest}
\end{equation}
\end{Prop}
\begin{proof}
Consider the product of all the non-zero terms on the left-hand side of 
Proposition \ref{prop-roots}. The removal of the zero terms, corresponding to 
those $a_{i}=1$ carry with them the removal of a power of $n$. 
\end{proof}

\begin{Example}
In example \ref{dilog} we have $\gamma_{a} =3$ and $\gamma_{b}=1$. The 
statement (\ref{interest}) states that 
\begin{equation}
n^{3} \times (1)_{n-1} (1)_{n-1} (1)_{n-1} = 
n \times (2)_{n-1} (2)_{n-1} (1)_{n-1}, 
\end{equation}
\noindent 
which is correct. 
\end{Example}

\begin{Thm}\label{eigen-characterize}
Assume $j=1$ and 
\begin{equation}
f(x) = x \sum_{k=0}^{\infty} \frac{(a_{1})_{k} \cdots (a_{p})_{k} }
{(b_{1})_{k} \cdots (b_{p})_{k} }
\end{equation}
\noindent 
is an eigenfunction for $U_{n}$ with eigenvalue $\lambda$. Then 
\begin{equation}
\lambda = n^{\gamma_{b} - \gamma_{a}},
\end{equation}
\noindent
with $\gamma_{a}, \, \gamma_{b}$ are defined in (\ref{interest}).
\end{Thm}
\begin{proof}
Put $k=1$ in the relation (\ref{crucial}) to obtain 
\begin{equation}
\lambda = \frac{(a_{1})_{n-1} \cdots (a_{p})_{n-1}}
{(b_{1})_{n-1} \cdots (b_{q+1})_{n-1}}.
\end{equation}
\label{crucial1}
\noindent
Now use (\ref{interest}) to conclude.
\end{proof}

The result above shows that the spectrum of $U_{n}$ satisfies 
\begin{equation}
\spun \subset \{ n^{a}: \in a \in \mathbb{Z} \}.
\end{equation}
\noindent
The examples below show that we actually have equality.

\begin{Example}
\label{polylog}
Let $i \in \mathbb{Z}$. Then the hypergeomeric series
\begin{equation}
f_{i}(x) := \sum_{k=1}^{\infty} k^{i}x^{k},
\end{equation}
\noindent
is an eigenfunction of $U_{n}$, with eigenvalue $n^{i}$. The dilogarithm 
corresponds to the case $i=-2$. 
\end{Example}

\begin{Thm}
The spectrum of $U_{n}$ is the set $\{ n^{i}: \, i \in \mathbb{Z} \}$. 
\end{Thm}

\medskip

We now discuss the eigenfunctions of $U_{n}$. Start with the sets 
\begin{equation}
\left\{ a_{i}-1, \frac{b_{i}-1}{n}, 
\frac{b_{i}}{n}, \cdots, \frac{b_{i}+n-2}{n} 
\,  \right\}_{i=1}^{p}  = 
\left\{ b_{i}-1, \frac{a_{i}-1}{n}, 
\frac{a_{i}}{n}, \cdots, \frac{a_{i}+n-2}{n} 
\right\}_{i=1}^{p}, 
\nonumber
\end{equation}
\noindent
that have appeared in Proposition \ref{prop-roots}. Recall that $\{ a_{i}, \, 
b_{i} \}$ are the parameters of the eigenfunction 
\begin{equation}
f(x) = x \sum_{k=0}^{\infty} \frac{(a_{1})_{k} \cdots (a_{p})_{k} }
{(b_{1})_{k} \cdots (b_{q+1})_{k} }x^{k}. 
\end{equation}
\noindent
We may assume that $a_{i} \neq b_{j}$, otherwise the term $(a_{i})_{k} = 
(b_{j})_{k}$ should be cancelled.  Observe that if we let $c_{i} = a_{i}-1$
and $d_{i} = b_{i}-1$, the basic set identity in Case 2 becomes the basic
set identity of Case 1, with $c_{i}$ instead of $a_{i}$ and $d_{i}$ instead of 
$b_{i}$. The reason why one cannot deduce $a_{i} = b_{i}$ is 
that some of the $a_{i}$ in Case 2 could be $1$, and the corresponding 
$c_{i}$ would vanish. This violates the basic assumption of Proposition 
\ref{moments}. 

We bypass this difficulty by defining 
\begin{equation}
a_{i}^{'} = \begin{cases}
           a_{i} \quad \text{ if } a_{i} \neq 1, \\
           2  \quad \text{ if } a_{i}  = 1, 
\end{cases}
\end{equation}
\noindent
and 
\begin{equation}
b_{i}^{'} = \begin{cases}
           b_{i} \quad \text{ if } b_{i} \neq 1, \\
           2  \quad \text{ if } b_{i}  = 1.
\end{cases}
\end{equation}
\noindent
The sets described above, obtained by replacing $a_{i}, b_{i}$ 
by $a_{i}', b_{i}'$
remain equal. In order to see this, observe that if $a_{1}=1$, then the 
elements containing $a_{1}$ are 
\begin{equation}
\left\{ \frac{a_{1}-1}{n}, \, \frac{a_{1}}{n}, \cdots, \frac{a_{1}+n-2}{n}
\right\}, 
\end{equation}
\noindent
are replaced by 
\begin{equation}
\left\{ \frac{a_{1}}{n}, \, \frac{a_{1}+1}{n}, \cdots, \frac{a_{1}+n-1}{n}
\right\}, 
\end{equation}
\noindent
that is equivalent to simply replacing $a_{1}$ by $a_{1}+1$.   Theorem 
\ref{eigenfunction-1} shows that $a_{i}' = b_{i}'$. If both $a_{i}$ and 
$b_{i}$ are equal to $1$ {\em or} both different from $1$, we conclude that 
$a_{i} = b_{i}$. This contradicts our original assumption. In the case 
$a_{i}'=b_{i}'$, with $a_{i}=1$, we conclude that $b_{i}=2$. Similarly, if 
$a_{i}'=b_{i}'$, with $b_{i}=1$, we obtain $a_{i} =2 $. Therefore, each 
$a_{i}  = 1$ produces the valid  pair $\{ 1, 2 \}$, and each $b_{i} = 1$
leads to $\{ 2, 1 \}$. Using the fact that there are no common parameters 
from top and bottom, we conclude that an eigenfunction must have the following 
structure:

\begin{Thm}
\label{thm-tai}
Assume 
\begin{equation}
f(x) = x \sum_{k=0}^{\infty} \frac{(a_{1})_{k} \cdots (a_{p})_{k} }
{(b_{1})_{k} \cdots (b_{q+1})_{k} }x^{k}
\end{equation}
\noindent
is an eigenfunction of  $U_{n}$, with the usual normalization 
$b_{q+1} = 1$. Then, either $a_{i}=2$ and $b_{i} =1$ for all $1 \leq i \leq p$, 
or $a_{i}=1$ and $b_{i} =2$ for all $1 \leq i \leq p$. 
\end{Thm}

In other words, if $f$ is an eigenfunction of any particular Hecke operator $U_n$, then
\begin{equation}
f(x) = _{a}F_{a-1}((\overbrace{1,1,1,\ldots,1}^{a}),(\overbrace{2,2,2,\ldots,2}^{a-1});x) = 
\sum_{k=0}^\infty \frac{1}{ k^a} x^k,
\end{equation}
or 
\begin{equation}
f(x) = _{a}F_{a-1}((\overbrace{2,2,2,\ldots,2}^{a}),(\overbrace{1,1,1,\ldots,1}^{a-1});x)=
\sum_{k=0}^\infty k^a x^k,
\end{equation}
where $a$ is any nonnegative integer.

\section{The Simultaneous Eigenfunctions of $U_{n}$ for all $n$}
In this section we completely characterize those hypergeometric functions which are simultaneous eigenfunctions of $U_n$ for all $n$, and give an application to the theory of completely multiplicative functions.   

What do the hypergeometric functions $f(x)=\isum{1}c_{k}x^{k}$ that are simultaneous eigenfunctions of all of the linear operators $U_{n}$  look like?  It turns out there is a simple answer:   they are precisely the polylogarithms and the rational functions
 $ \left(   x \frac{d}{dx}\right)^a  \left(    \frac{1 } { 1-x}  \right)$, as given by the following theorem.
 
 \begin{Thm}\label{simul-hyp}
Let 
\begin{equation}
f(x) = \isum{1}c_{k}x^{k}
\end{equation}
be a  hypergeometric  function with no constant term.
Then $f$ is a simultaneous eigenfunction for the set of all Hecke operators 
$\left\{U_{n}\right\}_{n=1}^{\infty}$ with respective eigenvalues $\left\{\lambda_{n}\right\}_{n=1}^{\infty}$
if and only if
\begin{equation}
f(x)=C\isum{1}k^{a}x^{k},
\end{equation}
with $a\in\ZZ$ and $C\in\CC$.  In other words, $f$ is a polylogarithm, or $f = \left(   x \frac{d}{dx}\right)^a  \left(    \frac{1 } { 1-x}  \right)$.      
\end{Thm}
\hfill  $\square$
\bigskip

Before proving this theorem, we prove an interesting corollary 
regarding a number theoretic fact concerning hypergeometric coefficients
 that are completely multiplicative functions of the summation index.

\begin{Cor}
The hypergeometric coefficient
\begin{equation}
c(n) = \frac{(a_{1})_{n-1}\cdots 
(a_{p})_{n-1}}{(b_{1})_{n-1}\cdots (b_{q})_{n-1}},
\end{equation}
 is a completely multiplicative function of $n$  if and only if it is of the form $Cn^{a}$
for some $a\in\ZZ$ and $C\in\CC$.
\end{Cor}
\begin{proof}
By the definition of the 
completely multiplicative function $c(n)$, we know that 
\begin{equation}
c(nk) = c(n)  c(k)
\end{equation}
for all $k, n \in \mathbb{N}$. 
This implies that 
\begin{equation}
f(x) = \isum{1}c(k)x^{k}
\end{equation}
is an eigenfunction of $U_{n}$, with eigenvalue $c(n)$. This holds for all 
$n$, thus it is a simultaneous eigenfunction. The result now follows from 
Theorem \ref{simul-hyp}.
\end{proof}

\bigskip

We recall that by definition $f$ is a simultaneous eigenfunction of all of the Hecke operators if and only if for all $n \in \mathbb{N}$,
\begin{equation}
U_{n}f = \lambda_{n}f.
\end{equation}

Treating $f$ as a power series without even considering its hypergeometric properties, we have
\begin{equation}
\isum{1}c_{nk}x^{k} = \lambda_{n}\isum{1}c_{k}x^{k}
\end{equation}
so that 
\begin{equation}
c_{nk} = \lambda_{n}c_{k}
\end{equation}
This is true for all $k\geq 1$, so we can set $k=1$ which gives us
\begin{equation}
c_{n} = \lambda_{n}c_{1}
\end{equation}
This, in turn, is true for all $n$, which proves the following.

\begin{Lem}\label{simul}
Let
\begin{equation}
f(x) = \isum{1}c_{k}x^{k}
\end{equation}
be a power series with no constant term that is a simultaneous eigenfunction 
for the set of operators $\left\{U_{n}\right\}_{n=1}^{\infty}$ with respective 
eigenvalues $\left\{\lambda_{n}\right\}_{n=1}^{\infty}$.  Then, for 
$i\geq 2$, we have
\begin{equation}
c_{i} = \lambda_{i}c_{1}
\end{equation}
and thus
\begin{equation}\label{gen-simul}
f(x) = c_{1}\isum{1}\lambda_{k}x^{k}.
\end{equation}
\end{Lem}
\hfill $\square$

\bigskip

\begin{proof} (of Theorem \ref{simul-hyp})
Let us suppose now that $f$, as defined in (\ref{gen-simul}), is a 
hypergeometric function.  By theorem \ref{eigen-characterize}, we have that 
\begin{equation}
\lambda_{n} = n^{\gamma_{b}-\gamma_{a}}
\end{equation}
where $\gamma_{b}$ and $\gamma_{a}$ are defined in \ref{interest}. This 
proves theorem \ref{simul-hyp}.
\end{proof}

\bigskip

In our formal hypergeometric notation, we see that any simultaneous 
eigenfunction must be the polylog
\begin{equation}
f(x) = _{a}F_{a-1}((\overbrace{1,1,1,\ldots,1}^{a}),(\overbrace{2,2,2,\ldots,2}^{a-1});x) = 
\sum_{k=0}^\infty \frac{1}{ k^a} x^k,
\end{equation}
or the rational function
\begin{equation}
g(x) = _{a}F_{a-1}((\overbrace{2,2,2,\ldots,2}^{a}),(\overbrace{1,1,1,\ldots,1}^{a-1});x)=
\sum_{k=0}^\infty k^a x^k,
\end{equation}
where $a$ is any nonnegative integer.
In conclusion, we see that if a hypergeometric function is an eigenfunction 
for a single operator $U_j$, then it is automatically a simultaneous 
eigenfunction for all of the Hecke operators $U_n$,  as $n$ varies over all
positive integers.    This situation lies in sharp contrast with the space of 
rational functions studied recently 
in   \cite{GilRobins},   \cite{Moll}, and \cite{Ed}.

\medskip

\no
{\bf Acknowledgments}. The work of the first author was partially funded 
by the National Science Foundation
$\text{NSF-DMS } 0409968$. The second author was partially supported by 
an SUG grant, Singapore.
The authors wish to thank Tai Ha for the proof of Theorem \ref{thm-tai}.

\bigskip

\bibliographystyle{plain}

\end{document}